\documentclass[11pt]{amsart}
\providecommand{\XLeftMargin}{2.5cm}
\providecommand{\XTopMargin}{2.5cm}
\providecommand{\XRightMargin}{2.5cm}
\providecommand{\XBottomMargin}{2.5cm}

\usepackage{amssymb,latexsym,amsmath,amsthm,amscd}
\usepackage[all]{xy}
\xyoption{arc}
\providecommand{\XLeftMargin}{2cm}
\providecommand{\XTopMargin}{1.8cm}
\providecommand{\XRightMargin}{2cm}
\providecommand{\XBottomMargin}{1.8cm}

\usepackage[left=\XLeftMargin,top=\XTopMargin,right=\XRightMargin,bottom=\XBottomMargin]{geometry}
\usepackage{graphicx}
\usepackage{todonotes}

\newlength\XXXMyLength\makeatletter
\@ifclassloaded{amsart}{\def\Xadjustleft[#1]{}
}{
\def\Xadjustleft[#1]{\setlength\XXXMyLength{#1}\ifnum\numexpr\leftmargin>\numexpr\XXXMyLength\hspace{-\XXXMyLength}\else\hspace{-\leftmargin}\fi}
}
\makeatother

\UseComputerModernTips

\theoremstyle{plain}
\newtheorem{thm}{Theorem}[section]
\newtheorem*{thm*}{Theorem}
\newtheorem{lemma}[thm]{Lemma}

\newtheorem*{lem*}{Lemma}
\newtheorem{prop}[thm]{Proposition}
\newtheorem*{prop*}{Proposition}
\newtheorem*{claim}{Claim}
\newtheorem{cor}[thm]{Corollary}
\newtheorem*{cor*}{Corollary}

\newtheorem*{conj*}{Conjecture}

\theoremstyle{definition}
\newtheorem{cons}[thm]{Construction}
\newtheorem*{cons*}{Construction}

\newtheorem*{df*}{Definition}
\newtheorem{nota}[thm]{Notation}
\newtheorem*{nota*}{Notation}

\newtheorem*{qu*}{Question}

\newtheorem{rmk}[thm]{Remark}
\newtheorem*{rmk*}{Remark}

\newtheorem*{ex*}{Example}

\newcommand{\bC}{\mathbb{C}}

\newcommand{\bG}{\mathbb{G}}

\newcommand{\bQ}{\mathbb{Q}}
\newcommand{\bR}{\mathbb{R}}

\DeclareMathOperator{\End}{End}

\DeclareMathOperator{\Res}{Res}

\DeclareMathOperator{\Tr}{Tr}

\DeclareMathOperator{\Gl}{GL}

\DeclareMathOperator{\Sl}{SL}
\DeclareMathOperator{\SO}{SO}

\DeclareMathOperator{\UU}{U}
\DeclareMathOperator{\SU}{SU}

\DeclareMathOperator{\SP}{Sp}

\DeclareMathOperator{\Orth}{O}

\DeclareMathOperator{\GSpin}{GSpin}

\newcommand{\injects}{\hookrightarrow}

\def\sumprime{\mathop{\sum{\raise3pt\hbox{${}'$}}}}

\makeatletter
\def\revddots{\mathinner{\mkern1mu\raise\p@
\vbox{\kern7\p@\hbox{.}}\mkern2mu
\raise4\p@\hbox{.}\mkern2mu\raise7\p@\hbox{.}\mkern1mu}}
\makeatother

\newcommand{\floor}[1]{\left\lfloor #1 \right\rfloor}

\newcommand{\comment}[1]{}

\newcommand{\GL}{\Gl}
\newcommand{\SL}{\Sl}
\newcommand{\Sp}{\SP}

\begin{document}

\title[Sub-Shimura Varieties for type $O(2,n)$]{Sub-Shimura Varieties for type $O(2,n)$}
\author{Andrew Fiori}

\address{Mathematics \& Statistics
612 Campus Place N.W.
University of Calgary
2500 University Drive NW
Calgary, AB, Canada
T2N 1N4}
\email{andrew.fiori@ucalgary.ca}
\thanks{A.F. gratefully acknowledge support from the Pacific Institute for Mathematical Sciences (PIMS)}

\begin{abstract}
We give a classification, up to consideration of component groups, of sub-Shimura varieties of those Shimura Varieties attached to orthogonal groups of signature $(2,n)$ over $\bQ$.
\end{abstract}


\subjclass[2010]{}

\maketitle

\section{Introduction}
\label{sec:intro}

The purpose of this article is to give a complete classification of the connected components of sub-Shimura varieties of the Shimura varieties associated to quadratic forms of signature $(2,n)$.
These sub-Shimura varieties are important algebraic cycles in the Chow ring and often have important applications. For example, Heegner cycles have an important role in the theory of Borcherds lifts (See for example \cite{Borcherd_inv}, \cite{Brunier_BP}, or \cite{KudlaCycles}) and the various reinterpretations of the theory of singular moduli (See for example \cite{grosszagier}, \cite{BruinierYang1}, \cite{bruinierkudlayang}, \cite{kudlarapoport} \cite{AndreattaGorenetal}).

This sort of problem has been looked at for other Shimura varieties. In particular, the work of \cite{SatakeHolomorphic} and \cite{AbdulaliTwists} looked at the sub-Shimura varieties of the Shimura varieties associated to Symplectic groups.

The main result of this paper is Theorem \ref{thm:mainresult} which completely characterizes the possible sub-Shimura varieties of the Shimura varieties attached to orthogonal groups of signature $(2,n)$ up to considerations of component groups. 

This paper is organized as follows:
\begin{itemize}
\item In Section \ref{sec:sym} we use the theory of Hermitian symmetric spaces and real lie groups to deduce information about the possible sub-Shimura varieties.
\item In Section \ref{sec:Shim} we apply the theory of algebraic groups over $\bQ$ to obtain further information about the groups associated to the sub-Shimura varieties and prove the main theorem.
\item In Section \ref{sec:Res} we state a natural generalization of our result to certain base changes from totally real fields.
\item In Section \ref{sec:Conc} we briefly discuss several problems that this work leaves open.
\end{itemize}

\section{Hermitian Symmetric Spaces}
\label{sec:sym}

In this section we will be working over $\bR$ and $\bC$ using the connection between Shimura varieties and Hermitian symmetric domains. We refer the reader to \cite{Helgason} and \cite{Milne_shimura} for the relevant background on the theory of Hermitian symmetric spaces and the connection to Shimura varieties. The key feature of this connection we shall use is that maps of Shimura varieties induce maps of Hermitian symmetric domains and that in both cases, the maps come from a map between associated algebraic groups.

The key feature of the case of orthogonal Hermitian symmetric spaces we shall use is that the relevant groups have real rank at most $2$ and that the vast majority of irreducible representations of simple non-compact real lie groups cannot land in an orthogonal group with such small real rank. 

Simple Hermitian symmetric spaces of the non-compact type are completely classified.
From the following classification we can already rule out sub-Shimura varieties of most types based simply on the real rank of the associated group. Ultimately we hope to determine for which $H$ below there can exist a map $\rho:H\rightarrow \SO(2,n)$ which induces a non-trivial map of Hermitian symmetric spaces.
\begin{center}
List of simple Hermitian symmetric spaces of non-compact type

\setlength\tabcolsep{2.5pt}.
\begin{tabular}{ | c | c | c | c }
\hline
	Group & Conditions & Real Rank  \\
\hline
     $\SU(p,q)$ & $1 \leq p \leq q$  &  $ p$ \\
     $\SO(2,n)$ & $n > 4 $  &   $2$ \\
     $\SO^\ast(n)$ & $n > 4$ &  $\floor{n/2}$ \\
     $\Sp(2n)$&   &  $n \ge 1$ \\
     $E_6$ & & $2$  \\
     $E_7$ &  & $3$ \\
\hline
\end{tabular}
\end{center}

\begin{rmk}
We have the following remarks about the classification:
\begin{itemize}
\item The group $\SO^\ast(1)$ is isomorphic to $\SO(2)$ which is compact.
\item $\SO^\ast(2)$ is not simple, the spin cover is associated to the Shimura curve over a totally real quadratic extension which is definite at one place and indefinite at the other. 
\item The spin cover of $\SO^\ast(3)$ is $\SU(1,3)$, and these give the same symmetric spaces. 
\item The spin cover of $\SO^\ast(4)$ is the same as that of $\SO(2,6)$, and gives the same symmetric space. 

\item We have that $\SU(2,2) \sim \SO(2,4)$ give the same symmetric space.
\item We have that $\Sp(4) \sim \SO(2,3)$ give the same symmetric space.
\item We have that $\SO(2,2) \sim \Sp(2)\times \Sp(2)$ is not simple, and give the same symmetric space.
\item We have that $\Sp(2) = \SU(1,1) \sim \SO(2,1)$ all give the same symmetric space.
\end{itemize}
\end{rmk}

The following Lemma gives us a stronger criterion than rank to characterize which real orthogonal groups a given representation can factor through.
\begin{lemma}
Let $\chi:\bG_m \rightarrow H$ be a cocharacter, let $\rho : H\injects \GL(V)$ be any representation.

Let $V^\infty$ be the subspace such that for $v\in V^\infty$ we have 
\[ \lim_{t\rightarrow \infty} \rho \circ \chi(t) (v)= 0. \]
If there exists a quadratic form on $V$ whose orthogonal group contains the image of $H$ its rank is at least $\lambda(\rho,\chi) = \dim(V^{\infty})$
\end{lemma}

The following Lie theoretic description of the quantity $\lambda(\rho,\chi)$ will provide a simpler method for establishing lower bounds on it.
\begin{prop}\label{prop:LieT}
Let $H$ be the a real form, with real rank at least one, of a simple lie group with complex root datum $(X,X^\vee,\Phi,\Phi^\vee)$.
Let $\tilde{\chi}$ be the element of the cocharacter lattice $X^\vee$ associated to map $\chi:\bG_m \rightarrow H$ defined over $\bR$.
Let $\rho$ be any representation of $H$ defined over $\bR$.

Then $\lambda(\rho,\chi)$ is precisely the number of weights $\alpha$ (counted with multiplicity) appearing in $\rho$ for which $(\alpha,\chi) > 0$.
\end{prop}

Denote by $\lambda(\rho)$ the maximum value of $\lambda(\rho,\chi)$ as we run over all $\chi$.

In light of Proposition \ref{prop:LieT}, we have the following trivial lower bound on $\lambda(\rho)$.
\begin{prop}
If the group $H$ is simple then the quantity $\lambda(\rho)$ is bounded below by the number of non-trivial $\bC$-irreducible factors in the representation $\rho$ multiplied by the real rank of the group.
\end{prop}

From the above, we immediately obtain also the strictly stronger lower bounds:
\begin{cor}
If $H' \subset H$ are simple, then the quantity $\lambda(\rho)$ is bounded below by the number of non-trivial $\bC$-irreducible factors in the representation $\rho|_{H'}$ times the real rank of $H'$.
\end{cor}

We shall now use the above Corollary to rule out most representations of the groups listed above. 

The following propositions follow immediately from the well known classification of representations of special unitary groups as well as the well known branching rules, see for example \cite{FH_rep}.
\begin{prop}\label{prop:branchSU2}
Consider the group $H=\SU(2,q)$, where $q\ge 2$ and the subgroup $H'=\SU(2,q-1)$
.
For every representation of $H$ except the standard representation, $\rho|_{H'}$ has at least two non-trivial $\bR$-irreducible factors.
\end{prop}

\begin{prop}\label{prop:standardSU2}
the standard representation for $\SU(p,q)$ over $\bR$ is reducible over $\bC$ and has two $\bC$-irreducible factors.
\end{prop}

\begin{cor}
The only non-trivial maps from a group $\SU(2,q)$ to a group $\SO(2,n)$ come from the isogeny $\SU(2,2)\sim \SO(2,4)$.
\end{cor}
\begin{proof}
Proposition \ref{prop:branchSU2} immediately rules out the existence of maps $\SU(2,q) \rightarrow \SO(2,n)$ not arising from the standard representation of $\SU(2,q)$ for $q>2$. (Note that for $q=2$, the rank of $\SU(2,1)$ is only $1$ and hence this is not an obstruction).

Proposition \ref{prop:standardSU2} immediately rules out the existence of maps $\SU(2,q) \rightarrow \SO(2,n)$ which do arise from the standard representation of $\SU(2,q)$ for all $q\ge 2$.

A simple check shows that $\SU(2,2)$ has two representations whose restriction to $\SU(2,1)$ is the standard representation, the second is the isogeny $\SU(2,2)\sim \SO(2,4)$ and that this is the only non-trivial representation of $\SU(2,2)$ landing in an $\SO(2,n)$ group.
\end{proof}

For simplicity we will consider $\SU(2,2)$ via its isogenous form $\SO(2,4)$. 

The following proposition is again an immediate consequence of the classification of representations of special unitary groups and their branching rules.
\begin{prop}\label{prop:branchSU1}
Consider $H=\SU(1,q)$ where $q\ge 2$ and $H'=\SU(1,1)$.
For every representation of $H$ except the standard representation, $\rho|_{H'}$ has at least one copy of the standard representation and one other non-trivial representation, hence at least three $\bC$-irreducible factors.
\end{prop}

The following corollary is an immediate consequence.
\begin{cor}
The only non-trivial maps from a group $\SU(1,q)$ to a group $\SO(2,n)$ for $q\ge 2$ come from the standard representation of $\SU(1,q)$.
\end{cor}

The following proposition is again an immediate consequence of the classification of representations of special orthogonal groups and their branching rules.
\begin{prop}
\begin{itemize}
\item
Consider $H=\SO(2,2n+1)$ with $n\ge2$ and $H'=\SO(2,2n)$.
For every representation other than the standard representation $\rho|_{H'}$ has at least two $\bR$-irreducible factors.
\item
Consider $H=\SO(2,2n+2)$ with $n\ge2$ and $H'=\SO(2,2n)$.
For every representation other than the standard representation $\rho|_{H'}$ has at least two $\bR$-irreducible factors.
\end{itemize}
Note that $\SO(2,2)$ has a second representation into $\SO(2,4)$ via the map $\SO(2,2) \sim \SO(1,2)\times\SO(1,2)$. The inclusion $\SO(1,2)\times\SO(1,2) \injects \SO(2,4)$ does not give a map of symmetric spaces.
\end{prop}
The following corollary is an immediate consequence.
\begin{cor}
The only non-trivial maps from a group $\SO(2,\ell)$ to a group $\SO(2,n)$ inducing a map of symmetric spaces for $\ell > 1$ come from the standard representation of $\SO(2,\ell)$.
\end{cor}

The following follows from an explicit check on the representations of $\SL_2$.
\begin{prop}
The group $\SU(1,1) = \SL(2) \sim \SO(2,1)$ admits three maps into orthogonal groups of rank $2$, namely:
\begin{itemize}
\item The inclusion of $\SU(1,1)$ into $\SO(2,2)$.
\item The symmetric squares representation to $\SO(2,1)$.
\item Two copies of the symmetric squares representation to $\SO(2,1)\times\SO(2,1) \injects  \SO(2,4)$; this final map does not give a map of symmetric spaces,
\end{itemize}
\end{prop}

The following is a consequence of work of Satake (see \cite{SatakeHolomorphic}) which gives a classification of possible sub-symmetric spaces for $\Sp(2n)$. 
\begin{prop}
There is no map from the Hermitian symmetric space associated to $E_6$ into any Hermitian symmetric space associated to $\Sp(2n)$ for any $n$.

There is a map from the Hermitian symmetric space for $\SO(2,\ell)$ into one associated to $\Sp(2n)$ for some $n$.

It follows that there is no map from the Hermitian symmetric space associated to $E_6$ into any Hermitian symmetric space associated to $\SO(2,\ell)$.
\end{prop}

The various lemmas and propositions above allow us to conclude:
\begin{lemma}
We can characterize those simple non-compact Hermitian symmetric spaces which embed into the Hermitian symmetric space associated to an orthogonal group of type $(2,n)$ as follows, they are up to isogeny:
\begin{itemize}
\item $\SU(1,q)$ or $\UU(1,q)$ where the map is given by the standard representation.

Note that the map $\UU(1,q) \rightarrow \UU(1,q) \times \SO(2)$ is up to isogeny of symmetric spaces equivalent to considering $\SU(1,q)\times \SO(2)$ or $\UU(1,q)\times \SO(2)$ and we have chosen to ignore compact factors.
\item $\SO(2,\ell)$ where the map is given by the standard representation.
\end{itemize}

Moreover, the only possible product of simple non-compact type Hermitian symmetric spaces which can embed into such a Hermitian symmetric space come from the natural inclusion of $SO(2,2)$ into $SO(2,n)$.
\end{lemma}

\begin{thm}
If $H\rightarrow \SO(2,n)$ gives rise to a map of Hermitian symmetric spaces of the non-compact type then up to isogeny:
\[ H \sim (H^{nc} \times H^{c}) \]
where  $H^c$ is compact, $H^{nc}$ is non-compact, and, up to isogeny, $H^{nc}$ is one of:
\[ \SO(2,\ell) \qquad \text{or} \qquad \SU(1,\ell) \qquad \text{or} \qquad  \UU(1,\ell). \]

Moreover, the map respects an orthogonal decomposition of the vector space into $V^{nc} \perp V^c$ where $V^{nc}$ has signature $(2,r)$ and $V^c$ has signature $(0,s)$ with $s+r=n$.
\end{thm}

\section{Shimura Varieties}
\label{sec:Shim}

In this section we will exploit the properties of algebraic groups over $\bQ$ to obtain conditions on the possible $\bQ$-structure of the groups whose Shimura varieties will admit maps into the Shimura varieties attached to orthogonal groups of signature $(2,n)$.

The key idea of this section is the following.
If $V$ is the vector space on which $\SO(2,n)$ has its standard representation we associate to any inclusion $\rho: H\rightarrow \SO(2,n)$ a sub-algebra of $\End(V)$.
We shall use in this section the connection between algebras with involution and algebraic groups to conclude that $\rho$ factors through a certain class of subgroup.
We finally exploit the results of the previous section to conclude that $\rho$ surjects onto our constructed group.

\begin{nota}
Let $(V,q)$ be a quadratic space of signature $(2,n)$ over $\bQ$.

Let $\rho: H\hookrightarrow \SO(V,q)$ be an inclusion of groups giving rise to a map of Shimura varieties.

Denote by $V^\rho$ the subspace of $\rho(H)$-fixed vectors.
\end{nota}

\begin{rmk}
We are assuming that this map is an inclusion.
This both limits the component group of the Shimura variety associated to $H$ but also has some impact on existence of compact $\bQ$-simple factors.
\end{rmk}

The following lemma gives a useful first step in the reduction.
\begin{lemma}
There exists a decomposition of $V=V^\rho \oplus (V^\rho)^\perp$ where $V^\rho$ denotes the $\rho$-fixed vectors of $V$.
Moreover, $V^\rho$ has signature $(0,m)$ for some $m$ and the map $\rho$ factors through:
\[ H \rightarrow \SO( (V^\rho)^\perp ) \rightarrow \SO(V). \]
\end{lemma}
From the above, we may without any loss of generality suppose that $V^\rho = \{ 0 \}$ and work inside the Heegner cycle associated to $V^\rho$.

\begin{cons}
Denote by $A$ the algebra generated by $\rho(H(\bQ))$ inside $\End(V)$, denote by $E$ the center of this algebra.

The adjoint involution $\sigma$ associated  to $q$  on $\End(V)$ stabilizes both $A$ and $E$ and hence these are both algebras with involution.

Let $F=E^\sigma$ be the sub-algebra of $\sigma$-fixed elements and $B$ be the centralizer of $F$. The algebra $B$ is again an algebra with involution, its center is $F$.

Decompose $F=\oplus F_i$ into a product of fields.
The idempotents for the $F_i$ decompose $V = \oplus V_i$, $E=\oplus E_i$, $A=\oplus A_i$ and $B=\oplus B_i$.

Then each $V_i$ is an $F_i$, $E_i$, $A_i$ and $B_i$-module.
Moreover, we have that $B_i = \End_{F_i}(V_i)$ is an algebra with involution and hence there exists an $F_i$-valued quadratic form $q_i$ on $V_i$ inducing the involution $\sigma|_{B_i}$.
By uniqueness, and polarization we can rescale $q_i$ so that $q|_{V_i} = \Tr_{F_i/\bQ}(q_i)$.

Packaging these together we find:
 \[ q = \Tr_{F/\bQ}(\oplus_i q_i) \]
and that the inclusion of $H$ into $SO(V)$ factors through the inclusion:
\[ \Res_{F/\bQ}( \times_i \Orth(q_i)) \rightarrow \times_i \Orth( \Tr_{F_i/\bQ}(q_i) ). \]
\end{cons}

The following claim can be checked after base change to $\bR$.
\begin{claim}
The algebra $F$ must be totally real, and $F$ has exactly one factor for which $\Tr_{F_i/\bQ}(q_i)$ is indefinite.
\end{claim}
\begin{proof}
Notice that the inclusion 
\[ \times_i \SO( \Tr_{F_i/\bQ}(q_i) )  \rightarrow \SO(q) \]
gives a map of symmetric spaces through which the inclusion of the Hermitian symmetric space for $H$ must factor.
It follows that each map $H(\bR) \rightarrow  \Orth( \Tr_{F_i/\bQ}(q_i) )(\bR)$ induces a map of symmetric spaces.

Thus, by the results of the previous section we can conclude that $H(\bR)$ has non-compact image in at most one of $\Orth( \Tr_{F_i/\bQ}(q_i) )(\bR)$. It follows that there is at most one (and hence exactly one) of $\Orth( \Tr_{F_i/\bQ}(q_i) )(\bR)$ which is non-compact, for if $H(\bR)$ had compact image in a non-compact $\Orth( \Tr_{F_i/\bQ}(q_i) )(\bR)$ this would induce a refinement of $F$ by virtue of the structure of the maximal compact subgroup of  $\Orth( \Tr_{F_i/\bQ}(q_i) )(\bR)$.

Next we must now rule out the possibility that this non-compact factor comes from a complex factor $F_j$ of $F$. Indeed the only possibilities for:
\[ \Res_{\bC/\bR}(\Orth(q_j)) \injects \Orth( \Tr_{\bC/\bR}(q_j) ) \]
which give quadratic forms with real rank at most $2$ are when $V_j$ has dimension $1$ or $2$ over $\bC$.

Suppose $V_j$ has dimension $1$ over $\bC=F_j$, so it has dimension $2$ over $\bR$. However, $H(\bR)$ factors through $\Res_{\bC/\bR}(\Orth(q_j))(\bR) = \{ \pm1 \} $. We can then readily check that regardless of whether the image of $H(\bR)$ is $\{1\}$ or $\{\pm1\}$ we will not have that $\bC$ is contained in the center of the algebra generated of the image.

In the second case, when $V_j$ has complex dimension $2$ over $\bC$, so it has dimension $4$ over $\bR$, we find that $\Res_{\bC/\bR}(\Orth(q_j))(\bR)  \simeq \{ \pm 1 \} \ltimes \bC^\times$ and the only maps of Hermitian symmetric domains which factor through this are discrete. It follows that $H(\bR)$ is compact and thus the image inside $\Orth( \Tr_{F_j/\bQ}(q_j) )(\bR)$ is the torus associated to a CM-algebra.
The algebra it generates inside $\End(V_j)$ can thus not be $\bC \times \bC$ with the exchange involution.

In particular, we have shown that $F_j\neq \bC$, so that there are no complex factors of $F$.
\end{proof}

We now consider two cases for the structure of the image of $H$ in the non-compact factor over $F_j$. The $2$ cases to consider are $E_j=F_j$, or $E_j$ a non-trivial $CM$-extension of $F_j$.

\begin{cons}
If $E_j$ is a non-trivial $CM$-extension of $F_j$.
In this case, the form $q_j$ on $V_j$ is the restriction of an $E_j$-valued Hermitian form on $V_j$, which we will denote $q_j'$.

If $E_j=F_j$, we find that through base change to $\bR$, we have:
\[ H(\bR) \rightarrow \SO(2,\ell)\times \SO(2+\ell) \times \cdots \times \SO(2+\ell) \]
and by the classification result of the first section, the projection onto the first factor $\SO(2,\ell)$ is surjective.

We conclude that $H = \Res_{F_j/\bQ}(\SO(q_j))$.

Similarly, in the second case we find that through base change to $\bR$, we have:
\[ H(\bR) \rightarrow U(1,\ell)\times U(1+\ell) \times \cdots \times U(1+\ell) \]
and by the classification result of the first section, the projection onto at least the $\SU(1,\ell)$ part of the first factor is surjective.

We conclude that $\Res_{F_j/\bQ}(\SU(q_j')) \subset H \subset \Res_{F_j/\bQ}(U(q_j'))$.
\end{cons}

We now consider the following modifications to our group $H$.

Consider the simply connected cover $\tilde{H}$ of $H$, decompose $\tilde{H} = \tilde{H}^{nc} \times \tilde{H}^c$ over $\bQ$ into a compact part and a non-compact part.
Denote by $H'$ the image in $H$ of $\tilde{H}^{nc}$.

\begin{claim}
Up to consideration of component groups the Shimura varietes associated to $H'$ and $H$ are isomorphic as are theier inclusions into the Shimura variety associated to $\SO(2,n)$.
\end{claim}

Because of the structure of $H'$, it admits no maps into $\bQ$-compact factors  $\Res_{\bC/\bR}(\Orth(q_j)) $.
In particular if we replace $H$ by $H'$, and repeat all of the above constructions, we obtain the same map of Shimura varieties (up to consideration of component groups) and we may assume that $F$ is a field.

The following theorem is now an immediate consequence of the various proceeding claims and propositions.
\begin{thm}\label{thm:mainresult}
Up to consideration of component groups, the Sub-Shimura varieties of $O(2,n)$ type Shimura varieties are of the following sort:

They arise from the restriction of scalars from a totally real field $F$ of 
\begin{itemize}
\item the Shimura variety attached to a quadratic form $q'$ which has signature $(2,\ell)$ at one place, and is negative definite at all other places.
\item the Shimura variety attached to a Hermitian form $q'$ associated to $CM$-extension $E$ of $F$ which has signature $(1,\ell)$ at one place of $F$, and is negative definite at all other places.
\end{itemize}
The only condition on the existence of these sub-Shimura varieties is that we can write:
\[ q = \Tr_{F/\bQ}(q')  \oplus q^{\rho} \qquad \text{or} \qquad q = \Tr_{E/\bQ}(q')  \oplus q^{\rho}. \]
\end{thm}

\begin{rmk}
The isomorphism class of the cycles described above depends on the choice of $F$ (respectively $E$) and $q'$, the isomorphism class of these (together with that of $q$) will determine that of $q^{\rho}$. Such a $q^{\rho}$ need not exist for any given pair $F,q'$ (respectively $E,q'$)  when the codimension is small.
Moreover, the isomorphism class need not determine the rational conjugacy class.

The ideas needed to study these questions are entirely different from those we are using here, and thus the questions are outside the scope of the current paper. We refer the interested reader to several places where work on this has been done.

Heegner cycles (and generalized Heegner cycles) correspond to taking $E=\bQ$. Various features of this case are studied in the work of \cite{KudlaCycles}.

Special points correspond to taking $E$ a $CM$-field and $q'$ a one dimensional quadratic space. A more thorough study of this case was undertaken in \cite{Fiori1}. Note that this is the only case we consider where the group $H(\bR)$ is compact.

The authors Thesis \cite{Fiori_PHD} and the preprint \cite{Fiori3} includes some discussion of how to compute the invariants of these restrictions of trace forms.
\end{rmk}

\begin{rmk}
We note that the Shimura varieties considered above (both the cycles we are considering and the variety in which we are embedding them) are generally disconnected. 

We note however, that by composition of maps every sub-Shimura variety of $\GSpin(2,n)$ leads to a sub-Shimura variety for $\SO(2,n)$ and conversely, if
$H\rightarrow \SO(2,n)$ gives a sub-Shimura variety  for $\SO(2,n)$ then $H\times_{\SO(2,n)} \GSpin(2,n) \rightarrow \GSpin(2,n) $ gives a sub-Shimura variety for $\GSpin(2,n)$.
Up to possibly subtle considerations of the component groups, we obtain essentially the same sub-Shimura varieties in each case.
Again, a study of these component groups, though a very interesting question both in terms of the action of the Galois group on them and even the simpler question of the possible sizes, requires entirely different techniques and is outside the scope of this paper.
\end{rmk}

\section{Restrictions of Scalars Shimura Varieties}
\label{sec:Res}

The following two theorems are essentially corollaries of Theorem \ref{thm:mainresult}, the proofs are left to the reader.

\begin{thm}
Let $F$ be a totally real field, let $q$ be an $F$ valued quadratic form which has signature $(2,\ell)$ at one place, and is negative definite at all other places.
The sub-Shimura varieties of the associated Shimura variety are of the following sort:

They arise from the restriction of scalars from a totally real field $F'$ which is an extension of $F$ of 
\begin{itemize}
\item the Shimura variety attached to a quadratic form $q'$ which has signature $(2,\ell)$ at one place, and is negative definite at all other places.
\item the Shimura variety attached to a Hermitian form $q'$ associated to $CM$-extension $E$ of $F'$ which has signature $(1,\ell)$ at one place of $F'$, and is negative definite at all other places.
\end{itemize}
The only condition on the existence of these sub-Shimura varieties is that we can write:
\[ q = \Tr_{F'/F}(q')  \oplus q^{\rho} \qquad \text{or} \qquad q = \Tr_{E/F}(q')  \oplus q^{\rho}. \]

The isomorphism class of such a cycle depends on the choice of $F'$ (respectively $E$) and $q'$; the isomorphism class of these with that of $q$ will determine that of $q^{\rho}$. When the codimension is small there is no guarentee that given $q$ and a pair $F',q'$ (respectively $E,q'$) that any form $q^\rho$ can actually be found.
\end{thm}

\begin{thm}
Let $E$ be a CM-field, let $q$ be a Hermitian form which has signature $(1,q)$ at one place, and is negative definite at all other places.
The sub-Shimura varieties of the associated Shimura variety are of the following sort:

They arise from the restriction of scalars from a CM-field $E'$, an extension of $E$, of 
 the Shimura variety associated to a Hermitian form $q'$ which has signature $(1,\ell)$ at one place of $E'$, and is negative definite at all other places.

The only condition on the existence of these sub-Shimura varieties is that we can write:
\[ q = \Tr_{E'/E}(q')  \oplus q^{\rho}. \]

The isomorphism class of such a cycle depends on the choice of $E'$ and $q'$; the isomorphism class of these with that of $q$ will determine that of $q^{\rho}$. Given $q$, $E'$ and $q'$ such a form $q^{\rho}$ need not exist when the codimension is small.
\end{thm}

\section{Further Questions}
\label{sec:Conc}

For $\SU(p,q)$ with $p>1$ and $\SO^\ast(n)$ with $n>4$, the idea of using branching rules and rank as was done here are adaptable to this setting. However, the number of representations that must be considered grows considerably; it is likely a more detailed study of the representations would be necessary and it is not immediately clear what results one can even hope for.

A careful inspection of the component group considerations which we have ignored would be worthwhile. The techniques we have applied here are not well suited for studying the effect of non-$\bQ$-simple compact factors. In particular all representations of compact groups can map into a sufficiently large orthogonal group.

As discussed in the paper, determining the precise local/global conditions under which a cycle of a given rational isomorphism class will embed into the Shimura variety with a specified rational isomorphism class is worthy of some study, as is the more refined question of describing precisely the rational (or even integral) conjugacy classes of these embeddings.

\section*{Acknowledgments}

The author would like to thank Prof. Eyal Goren for suggesting the problem studied in this paper. They would also like to thank Majid Shahabi for their help in proofreading this article.

{}\ifx\XMetaCompile\undefined

\providecommand{\MR}[1]{}
\providecommand{\bysame}{\leavevmode\hbox to3em{\hrulefill}\thinspace}
\providecommand{\MR}{\relax\ifhmode\unskip\space\fi MR }
\providecommand{\MRhref}[2]{  \href{http://www.ams.org/mathscinet-getitem?mr=#1}{#2}
}
\providecommand{\href}[2]{#2}

\end{document}
{}\fi